\documentclass[leqno]{amsart}
\usepackage{amsmath}
\usepackage{amssymb}
\usepackage{amsthm}
\usepackage{enumerate}
\usepackage[mathscr]{eucal}
\usepackage{graphicx}
\theoremstyle{plain}
\newtheorem{theorem}{Theorem}[section]
\newtheorem{prop}[theorem]{Proposition}
\newtheorem{cor}{Corollary}[theorem]
\newtheorem{lemma}{Lemma}[section]

\theoremstyle{definition}
\newtheorem{definition}{Definition}[section]
\newtheorem{remark}{Remark}[section]

\usepackage[pagewise]{lineno}
\begin{document}
\title[On the numerical range of  operators on some special Banach spaces ]{On the numerical range of  operators on some special Banach spaces }
\author{  Kalidas Mandal, Aniket Bhanja, Santanu Bag  and  Kallol paul }

\newcommand{\acr}{\newline\indent}

\address[Mandal]{Department of Mathematics\\ Jadavpur University\\ Kolkata 700032\\ West Bengal\\ INDIA}
\email{kalidas.mandal14@gmail.com}

\address[Bhanja]{Department of Mathematics\\ Vivekananda College Thakurpukur\\ Kolkata\\ West Bengal\\India\\ }
\email{aniketbhanja219@gmail.com}

\address[Bag]{Department of Mathematics\\ Vivekananda  College for Women, Barisha \\ Kolkata  \\ West Bengal\\ INDIA}
\email{santanumath84@gmail.com}

\address[Paul]{Department of Mathematics\\ Jadavpur University\\ Kolkata 700032\\ West Bengal\\ INDIA}
\email{kalloldada@gmail.com}

\thanks{} 


\begin{abstract}
	
	The numerical range of a bounded linear operator on a complex Banach space need not be convex unlike that on a Hilbert space. The aim of this paper is to study operators $T$ on $ \ell^2_p $  for which the numerical range is convex.  We also obtain a nice relation between $V(T)$ and $ V(T^t)$ considering $ T \in \mathbb{L} (\ell_p^2) $ and $ T^t \in \mathbb{L} (\ell_q^2) ,$ where $T^t$ denotes the transpose of $T$  and $p$ and $q$ are conjugate real numbers i.e., $ 1 <p,q< \infty $ and $ \frac{1}{p}+\frac{1}{q}=1.$

\end{abstract}

\maketitle
\section{Introduction.}

\medskip

The numerical range of a bounded linear operator on a real or complex Hilbert space is a convex  subset of the scalar field by the well known result of Toeplitz-Hausdroff theorem \cite{H}. On real Banach spaces of dimension ($ \geq 2$) the numerical range of a bounded linear operator  is always  convex \cite{BD}. For operators defined on a complex Banach space the numerical range is connected but not necessarily convex. The purpose of this paper is to find classes of operators on certain special Banach spaces for which the numerical range is convex. Let us first introduce some notations and terminologies.

Let $\mathbb{X} $ and $\mathbb{H}$  denotes complex Banach and  Hilbert spaces respectively.  Let $B_{\mathbb{X}} = \{x \in \mathbb{X} \colon \|x\| \leq 1\}$ and
$S_{\mathbb{X}} = \{x \in \mathbb{X} \colon \|x\|=1\}$ be the unit ball and the unit sphere of $\mathbb{X}$, respectively. Let $\mathbb{L}(\mathbb{X})$ ( $ \mathbb{L}(\mathbb{H}) ) $ denote the  space of all bounded linear operators defined on $\mathbb{X}$ ( $ \mathbb{H}$ ), endowed with the usual operator norm. 

Let $ T \in  \mathbb{L}(\mathbb{H}) $, then  the numerical range, $ W(T) $ and the numerical radius, $w(T) $ are defined as : 
\[W(T)= \{\langle Tx, x\rangle: x\in \mathbb{H}, \|x\|=1 \},\]
\[w(T)= \sup\{|\lambda|: \lambda \in W(T)\}.\] 
For further information on the properties of numerical range and related things one can see \cite{ Ba, B, GR,H,T}.
The  structure of inner product inherently present in Hilbert space helps to study the geometric properties of both $H$ and $ \mathbb{L}(\mathbb{H})$ and that motivated  Lumer \cite{L} to introduce semi inner products (s.i.p.) in normed spaces. This  notion of s.i.p. and its applications in the study of Banach space geometry has grown over the years with contribution from  many mathematicians \cite {D,G}.
This is the proper time to write down the well known definition of s.i.p. in a normed space.

\begin{definition}[\cite{L,G}]
	Let $\mathbb{X}$ be a normed linear space. A function [ , ]: $\mathbb{X}\times \mathbb{X} \rightarrow \mathbb{K}(=\mathbb{C},\mathbb{R})$ is a semi-inner-product if and only if for any $\lambda \in \mathbb{K}$ and for any $x,y,z \in \mathbb{X}$, it satisfies the following properties:\\
	(i) $[x+y , z] = [x , z] +[y , z]$,\\
	(ii) $[\lambda x ,y] = \lambda [x , y]$ \& $[x , \lambda y] = \overline \lambda [x , y]$,\\
	(iii) $[x , x]>0$, when $x\neq 0$,\\
	(iv) $|[x , y]|^2\leq [x , x][y , y]$.\\
	
\end{definition} 
For every normed linear space, there exists a s.i.p. $[~,~]$ which is compatible with the norm (see \cite{G} ), i.e., $[x,x]=\|x\|^2 $ for all $x\in \mathbb{X}.$ Given a normed linear space, there exists many s.i.p.s compatible with the norm. In case of smooth space there exists exactly one s.i.p. compatible with the norm and that s.i.p. is the inner product if it is a Hilbert  space. Let us recall that a normed space $\mathbb{X}$ is said to be smooth if for each $ x \in S_{\mathbb{X}} $ there exists a unique supporting functional at $x$ i.e., there exists a unique linear functional $x^*$ such that $ x^*(x) =1 $ and $  \|x^*\| = 1.$\\
For $ T \in  \mathbb{L}(\mathbb{X}) $, the numerical range $ V(T) $ and the numerical radius $v(T) $ are defined as : 
\[V(T)=\{[Tx, x]:[x,x]=1\},  v(T)= \sup\{|\lambda| : \lambda \in V(T)\}.\]
These are equivalent to the following definition: 
\[ V(T) = \{ x^*(Tx) : \|x \| =1, \|x^*\| =1, x^*(x) =1\}, v(T)= \sup\{|\lambda| : \lambda \in V(T)\}.\]
We denote boundary of $V(T)$ by $\delta V(T).$

There is not much literature available on the structure of $V(T)$ as it is not necessarily convex. The purpose of this paper is to present a class of operators on $\ell_p^2 ( 1 < p < \infty) $  for which the numerical range $V(T)$ is convex. We also obtain a nice relation between $V(T)$ and $ V(T^t)$ considering $ T \in \mathbb{L} (\ell_p^2) $ and $ T^t \in \mathbb{L} (\ell_q^2) ,$ where $T^t$ denotes the transpose of $T$ and $p$ and $q$ are conjugate real numbers i.e., $ 1 <p,q< \infty $ and $ \frac{1}{p}+\frac{1}{q}=1.$

\section{Main Results}   

We begin with the following elementary results on numerical range and numerical radius, the proof of which follows from definitions. 

\begin{prop}\label{prop}\cite{L}
	Let $T_1, T_2$ be any operators on semi-inner-product space $\mathbb{X}$, $I$ be the identity operator and $\alpha , \beta \in \mathbb{K}(=\mathbb{C},\mathbb{R})$. Then we have\\
	$(i)$ $v(T_1)\leq \|T_1\|,$\\
	$(ii)$ $V(\alpha I+\beta T_1 )=\alpha +\beta V(T_1),$\\
	$(iii)$ $ V(T_1+ T_2)\subset V(T_1)+ V(T_2).$
\end{prop}
	Let $\ell_p, 1 < p < \infty$ be the complex normed space of all $p-$summable sequences, with the norm\\
	\[\|\tilde{x}\|_p=\bigg (\sum_{k=1}^{\infty} |x_k|^p\bigg)^\frac{1}{p},\]
	where $\tilde{x}=(x_1,x_2,x_3,....)\in \ell_p.$  Given any two elements $\tilde{x}=(x_1,x_2,x_3,....), \tilde{y}=(y_1,y_2,y_3,....)\in \ell_p, $ the unique semi-inner-product can be easily seen to be 
	\begin{eqnarray*}
		 [\tilde{x},\tilde{y}]_p & =& \frac{\sum_{k=1}^{\infty}x_k\overline{y_k}|y_k|^{p-2}}{\|\tilde{y}\|_p^{p-2}}, \tilde{y} \neq 0\\
		 & = & 0, \tilde{y}=0.
		 \end{eqnarray*}
	where $\overline{y_k}$ denotes the complex conjugate of $y_k,$ that generates the norm $\|.\|_p.$ We denote $n$-dimensional $\ell_p$ space by $\ell^n_p. $ \\
	We note that if $ D $ is  an $n\times n$ diagonal matrix with complex entries $d_{ii},$ then $V(D)$  is the convex hull of its entries  and $v(D)=\max \{|d_{ii}|:1\leq i \leq n\}$.

We now prove first result which states that numerical range of a nilpotent operator on $\ell^2_p $ is convex.

\begin{theorem}\label{2-d}
Let $T\in \mathbb{L}(\ell^2_p)$ be such that $ T =
\begin{pmatrix}
0 & 1 \\ 
\rule{0em}{3ex}0 & \hphantom{-} 0
\end{pmatrix} $. Then  $V(T)$ is a circular disc with center at origin and radius $v(T)= (\frac{1}{p})^{\frac{1}{p}}(\frac{1}{q})^{\frac{1}{q}}.$
\end{theorem}
\begin{proof} Let  $x=(x_1 , x_2) \in S_{\ell^2_p}$, 
 $x_1 ={\cos^{\frac{2}{p}} \theta} ~e^{i\alpha}$ and $x_2 ={\sin^{\frac{2}{p}} \theta} ~e^{i\beta}$, where $\theta , \alpha , \beta \in [0, 2\pi].$  Then 
$[Tx,x]_p =x_{2} \overline {x_{1}} |x_{1}|^{p-2}={\sin}^{\frac{2}{p}}\theta  {\cos^{\frac{2}{q}}\theta}  e^{i\phi}$, where $\phi=\beta-\alpha$ and $q=\frac{p}{p-1}.$ 
Therefore, real part of $[Tx,x]_p=\Re {[Tx,x]_p}={\sin}^{\frac{2}{p}}\theta  {\cos^{\frac{2}{q}}\theta} \cos \phi$
 and 
imaginary part of $[Tx,x]_p=\Im {[Tx,x]_p}={\sin}^{\frac{2}{p}}\theta  {\cos^{\frac{2}{q}}\theta} \sin \phi.$
Hence 
	$$(\Re {[Tx,x]_p})^2+(\Im {[Tx,x]_p})^2=  ({\sin}^{\frac{2}{p}}\theta  {\cos^{\frac{2}{q}}\theta})^2.$$
	These are  equations of concentric circles with center at origin. Therefore, the numerical range of $T$ is the disc with radius 
	\[v(T)=\sup_{\theta\in [0,2\pi]}{\sin}^{\frac{2}{p}}\theta  {\cos}^{\frac{2}{q}}\theta \]
	By a simple calculation, it can be seen that $v(T)= (\frac{1}{p})^{\frac{1}{p}}(\frac{1}{q})^{\frac{1}{q}}.$\\
	\end{proof}

By proposition (ii) of \ref{prop} and Theorem \ref{2-d}, we get the following two  theorems.    

 	\begin{theorem}
Let $T\in \mathbb {L}(\ell_p^2)$ be such that $T = \begin{pmatrix}
		\alpha & \beta  \\
		0 & \alpha \\
	\end{pmatrix}$, where $\alpha ,\beta \in \mathbb{K}$. Then $ V(T) $ is a disc with center at $\alpha$ and radius $|\beta|(\frac{1}{p})^{\frac{1}{p}}(\frac{1}{q})^{\frac{1}{q}}.$ 
\end{theorem}
 
 \begin{theorem}
 	Let $T\in \mathbb{L}(\ell^n_p)$ be such that  $T=(t_{ij})_{n\times n}$ where
 	\begin{eqnarray*}
 			t_{ij} & = & 1, (i,j) =  (i_0,j_0)~~and~~ i_0\neq j_0,\\
         	t_{ij} & = & 0, (i,j) \neq  (i_0,j_0).
	\end{eqnarray*}
	Then $V(T)$ is a disc with center at origin and radius $v(T)= (\frac{1}{p})^{\frac{1}{p}}(\frac{1}{q})^{\frac{1}{q}}.$
 \end{theorem}

	We now consider the class of operators of the form $T = \begin{pmatrix}
	1 & b  \\
	0 & -1 \\
	\end{pmatrix}$ and proceed step by step to show that the numerical range $V(T)$ is convex. First we prove the following theorem.

\begin{theorem}\label{boundary}
	Let $T\in \mathbb{L}(\ell^2_p)$ be such that $T = \begin{pmatrix}
		1 & b  \\
		0 & -1 \\
	\end{pmatrix}.$ \\
	Then \[V(T)=\bigcup_{\theta\in [0,2\pi]} C(\theta),\] 
	where for each $\theta \in [0, 2 \pi]$, 
	\[C(\theta)= \{(x,y):(x-\cos2\theta)^2+y^2= R^2(\theta)\},\] 
	$$ \mbox{with}~~R(\theta)=|b|(\sin^2 \theta)^\frac{1}{p}(\cos^2\theta)^\frac{1}{q}.$$ 
\end{theorem}
\begin{proof}
Let  $x=(x_1 , x_2) \in S_{\ell^2_p}$. 
Suppose  $
x_1 =(\cos\theta)^{\frac{2}{p}} e^{i\phi_1}, ~x_2 =(\sin\theta)^{\frac{2}{p}} e^{i\phi_2}, $ where $ \theta,~\phi_1,~\phi_2 \in[0,2\pi].$  Then  $[Tx,x]_p=|x_1|^p+bx_2\overline{x_1}|x_1|^{(p-2)}-|x_2|^p =\cos2\theta+|b|(\sin\theta)^\frac{2}{p}(\cos\theta)^\frac{2}{q}e^{i(\phi_2-\phi_1+\theta_b)},$  where  $  b=|b|e^{i\theta_b}.$
	Let $[Tx,x]_p=h+ik,$ where 
	$h=\cos2\theta+|b|(\sin\theta)^{\frac{2}{p}}(\cos\theta)^\frac{2}{q}\cos\psi$,
	$k=|b|(\sin\theta)^{\frac{2}{p}}(\cos\theta)^\frac{2}{q}\sin\psi$
	and  $\psi=(\phi_2-\phi_1+\theta_b).$
	Eliminating $\psi$ from $h$ and $k$, we get 
\[(h-\cos2\theta)^2+k^2=|b|^2(\sin\theta)^\frac{4}{p}(\cos\theta)^\frac{4}{q}=R^2(\theta)~\mbox{(say)}.\]
  For each  $\theta \in [0, 2 \pi],$ write $C(\theta)= \{(x,y):(x-\cos2\theta)^2+y^2= R^2(\theta)\}.$ 
Therefore, from definition it follows that \[V(T)=\bigcup_{\theta\in [0,2\pi]} C(\theta).\] 

	\end{proof}

 We are now in a position to show that the numerical range of a particular class of operators is convex.
\begin{theorem}\label{convex}
	Let $T\in \mathbb{L}(\ell^2_p)$ be such that $T = \begin{pmatrix}
		1 & b  \\
		0 & -1 \\
	\end{pmatrix}.$ Then $V(T)$ is convex set. 
\end{theorem}
\begin{proof}
By Theorem \ref{boundary}, $V(T)=\bigcup_{\theta\in [0,2\pi]} C(\theta),$
where $C(\theta)= \{(x,y):(x-\cos2\theta)^2+y^2= R^2(\theta)=|b|^2(\sin^2 \theta)^\frac{2}{p}(\cos^2\theta)^\frac{2}{q}\} .$
Let $u\in C(\theta_1)$ and $v\in C(\theta_2).$ 
Then \begin{eqnarray*}
	 u &=& (\cos 2\theta_1+R(\theta_1)\cos \psi_1,R(\theta_1)\sin \psi_1)\\
	 v &=& (\cos 2\theta_2+R(\theta_2)\cos \psi_2,R(\theta_2)\sin \psi_2),
	 \end{eqnarray*} 
	for some $\psi_1,\psi_2 \in [0,2\pi].$\\
	Let $t\in [0,1]$ and $(1-t)u +tv=(\alpha, \beta).$ Then
	\begin{eqnarray*} 
		\alpha&=&(1-t)\cos 2\theta_1+t\cos 2\theta_2+(1-t)R(\theta_1)\cos \psi_1+tR(\theta_2)\cos \psi_2\\
		\beta&=&(1-t)R(\theta_1)\sin \psi_1+tR(\theta_2)\sin \psi_2 .
	\end{eqnarray*}
	We want to show that  $(\alpha, \beta) $ lies on $ C(\theta)$ for some $\theta\in [0,2\pi]$, i.e., there exists $\theta \in [0,2\pi] $ such that $(\alpha-\cos2\theta)^2+\beta^2= R^2(\theta)=\frac{|b|^2}{4}(1-\cos 2\theta)^\frac{2}{p}(1+\cos2\theta)^\frac{2}{q}.$
Consider$ f(\theta)=(\alpha-\cos2\theta)^2+\beta^2-\frac{|b|^2}{4}(1-\cos 2\theta)^\frac{2}{p}(1+\cos2\theta)^\frac{2}{q}.$\\
\noindent Our claim is established if we can exhibit some $\theta \in [0,2\pi]$ for which $f(\theta)=0.$ Now $f(0)=(\alpha-1)^2+\beta^2 \geq 0.$
Choose $\theta_0\in [0,2\pi]$ such that $\cos2\theta_0 = (1-t)\cos 2\theta_1+t\cos 2\theta_2.$ Then it can be seen that
$f(\theta_0)=[(1-t)R(\theta_1)\cos \psi_1+tR(\theta_2)\cos \psi_2)]^2+[(1-t)R(\theta_1)\sin \psi_1+tR(\theta_2)\sin \psi_2]^2-R^2(\theta_0) = (1-t)^2R^2(\theta_1)+t^2R^2(\theta_2)+2t(1-t)R(\theta_1)R(\theta_2)\cos (\psi_1-\psi_2)-R^2(\theta_0).$
So,  $f(\theta_0)\leq [(1-t)R(\theta_1)+tR(\theta_2)]^2-R^2(\theta_0).$\\ 
 
Let $g(z)=\frac{|b|}{2}(1-z)^\frac{1}{p}(1+z)^\frac{1}{q}.$
Clearly, $g(z)\geq 0 $ for all $z\in [-1,1]$ and $g(z)$ is continuous on $[-1,1].$ It is easy to show that $g(z)$ is concave in $[-1,1].$
Therefore $f(\theta_0)\leq [(1-t)g(\cos2\theta_1)+tg(\cos2\theta_2)]^2-[g(\cos2\theta_0)]^2\leq 0$.
Hence $f$ has a solution in $[0,2\pi]$ as $f$ is continuous. Thus $V(T)$ is convex.
\end{proof}

Using  Theorem \ref{convex}, we have the following corollary.
\begin{cor}
	Let $T\in \mathbb{L}(\ell^2_p)$ be such that $T = \begin{pmatrix}
	a & b  \\
	0 & d \\
	\end{pmatrix},$ where $a,b,d\in \mathbb{K}.$ Then $V(T)$ is a convex set.
\end{cor}
\begin{proof}
 If $a=d$, then by Theorem 2.3, $V(T)$ is convex set.\\ 
 If $a\neq d$, then consider $T' = \begin{pmatrix}
	1 & \frac{2b}{a-d}  \\
	0 & -1 \\
	\end{pmatrix}$. Now $T$ can be written as 
	\[T=\frac{a+d}{2}I+\frac{a-d}{2}T'.\]
	Since $V(T')$ is convex set by Theorem \ref{convex}, then by proposition \ref{prop}, we can conclude that $V(T)$ is a convex set.
\end{proof}

Next, we obtain the equation of boundary of numerical range of the same class of operators.

\begin{theorem}
	Let $T\in \mathbb{L}(\ell^2_p)$ be such that $T = \begin{pmatrix}
	1 & b  \\
	0 & -1 \\
	\end{pmatrix}$. Then the parametric equation of the boundary of $V(T)$ is given by 
	\[x=\cos2\theta +f(\theta)\]
	\[y^2=R^2(\theta)-f^2(\theta),\]
	where
	$f(\theta)=\frac{|b|^2}{2}(\tan^2\theta)^{(\frac{2-p}{p})}[\frac{1}{p}\cos^2\theta-\frac{1}{q}\sin^2\theta]
	,~ R(\theta)=|b|(\sin^2 \theta)^\frac{1}{p}(\cos^2\theta)^\frac{1}{q}~ \mbox{with}~\\\frac{1}{p}+\frac{1}{q}=1 $  and  $\theta\in [0,2\pi]\setminus\{0,\frac{\pi}{2},\pi, \frac{3\pi}{2},2\pi\}.$\\
	Moreover
	$ (\cos2\theta+f(\theta),\pm \sqrt{R^2(\theta)-f^2(\theta)})\in \delta{V(T)} ~\mbox{ if and only if }~ \tan^{\frac{4}{q}} \theta\geq \frac{|b|^2}{4}(\frac{1}{p}-\frac{1}{q}\tan^2\theta)^2.$
\end{theorem}
\begin{proof}
	By Theorem \ref{boundary}, it is clear that the boundary of the numerical range of $T$ is the envelope of the family of the circle $ C(\theta)$, where
	\[C(\theta)= \{(x,y):(x-\cos2\theta)^2+y^2= R^2(\theta),~~\mbox{where}~\theta\in [0,2\pi]\}.\]	
	Now differentiating 
	\[(x-\cos2\theta)^2+y^2=R^2 (\theta)\]
	partially with respect to $\theta$ we get,
	\[2(x-\cos2\theta)2\sin2\theta=2R(\theta) R'(\theta).\]	
	Since $ R'(\theta)=|b|\sin{2\theta}[\frac{1}{p}(\cot\theta)^\frac{2}{q}-\frac{1}{q}(\tan\theta)^\frac{2}{p}] ,$ we get
	\[x=\cos2\theta+\frac{|b|^2}{2}(\tan^2\theta)^{(\frac{2}{p}-1)}[\frac{1}{p}\cos^2\theta-\frac{1}{q}\sin^2\theta].\]
	Note that $x$ is undefined if $\frac{2}{p}-1<0$ and $\theta=0, \pi,2\pi$ or $\frac{2}{p}-1>0$ and $\theta=\frac{\pi}{2},\frac{3\pi}{2}$. In these cases, the coordinate $(x,\pm y)$ does not contribute to the boundary of the numerical range of $T$ as all these conditions gives $R(\theta)=0.$\\
	Also $y$ is undefined if $(R(\theta))^2<(f(\theta))^2~\textit{i.e.}, \tan^{\frac{4}{q}} \theta< \frac{|b|^2}{4}(\frac{1}{p}-\frac{1}{q}\tan^2\theta)^2,$ where 
	$f(\theta)=\frac{|b|^2}{2}(\tan^2\theta)^{(\frac{2-p}{p})}[\frac{1}{p}\cos^2\theta-\frac{1}{q}\sin^2\theta]$.
	Then 
	\[x=\cos2\theta+f(\theta)\]
	\[y^2=R^2(\theta)-f^2(\theta).\]
\end{proof}

\begin{remark}\label{rem1}
	(i) We note that $V(T)$ is the union of family of circles with center at $(\cos2\theta,0)$, i.e., $V(T)$ is the collection of circles whose center lie on $X-$ axis. Therefore, $V(T)$ is symmetric about $X-$ axis. We note the following:\\
	(ii ) If $(x,\pm y)\in \delta{V(T)} ~\mbox{then }~ (f'(\theta)-2\sin2\theta)\neq 0.$\\
	(iii)  If $p=2$, by Theorem \ref{boundary}, we get $x=\cos2\theta(1+\frac{|b|^2}{4})$ and $y^2=\frac{|b|^2}{4}[(\sin2\theta)^2-\frac{|b|^2}{4}(\cos2\theta)^2].$ Then eliminating $\theta$, we get $\frac{x^2}{1+\frac{|b|^2}{4}}+\frac{y^2}{\frac{|b|^2}{4}}=1, $ which is the equation of boundary of the numerical range $ W(T)$ for operator $T = \begin{pmatrix}
	1 & b  \\
	0 & -1 \\
	\end{pmatrix}$ acting on  complex Hilbert space. 
\end{remark}

Next, we give an upper bound for numerical radius of the operator $T = \begin{pmatrix}
		1 & b  \\
		0 & -1 \\
	\end{pmatrix}$. 

\begin{theorem}
	Let $T\in \mathbb{L}(\ell^2_p)$ be such that $ T =
	\begin{pmatrix}
	1 & b \\ 
	0 & -1
	\end{pmatrix}.$ Then $v(T)\leq 1+|b|(\frac{1}{p})^{\frac{1}{p}}(\frac{1}{q})^{\frac{1}{q}}.$
\end{theorem}
\begin{proof}
Let $x=(x_1,x_2)\in S_{\ell_p^2}$.
Then $[Tx,x]_p=|x_1|^p+bx_2\overline{x_1}|x_1|^{(p-2)}-|x_2|^p$.\\

Therefore, 
\begin{eqnarray*}
v(T)&=&\sup_{x\in S_{\ell_p^2}}\bigg||x_1|^p+bx_2\overline{x_1}|x_1|^{(p-2)}-|x_2|^p\bigg|\\
&\leq & \sup_{x\in S_{\ell_p^2}}\big[|x_1|^p+|b||x_2||x_1|^{(p-1)}+|x_2|^p\big]\\
&=& 1 +|b|(\frac{1}{p})^{\frac{1}{p}}(\frac{1}{q})^{\frac{1}{q}}.
 \end{eqnarray*}
\end{proof}

The numerical range of an operator $T\in\mathbb{L}(\ell^2_p),$ may not be convex. Here we find a class of operators $T\in\mathbb{L}(\ell^2_p)~,~p\neq1,2,\infty$ for which $V(T)$ is not convex. 
\begin{theorem}
Let $T\in\mathbb{L}(\ell^2_p )~,~p\neq1,2,\infty$ be such that 
\[T=\begin{pmatrix}
	a & b  \\
	c & d \\
	\end{pmatrix}; a,d\in\mathbb{R}\setminus \{0\}, b,c\neq 0\]
	and\\ 
	$(i)~~~a+d=0,$\\
	$(ii)~~|b|=|c|,$\\
	$(iii)\Re{(b)}\Im{(c)}+\Re{(c)}\Im{(b)}=0.$\\	
 Then $V(T)$ is not convex.
\end{theorem}
\begin{proof}
By given condition, we can consider $T\in\mathbb{L}(\ell^2_p)$ such that \[T=\begin{pmatrix}
	a & b_1+ib_2  \\
	c_1+ic_2 & -a \\
	\end{pmatrix},\] where $a,b_1,b_2,c_1,c_2 \in \mathbb{R}$ and $b_1c_2+b_2c_1=0.$ \\
	Let $x=(x_1,x_2)\in S_{\ell^2_p}$ and $x_1=r e^{i\phi_{1}},x_{2}=s e^{i\phi_{2}}$, where $r,s\geq 0 ,\phi_{1},\phi_{2}\in [0,2\pi]$.\\
	Then \begin{eqnarray*}
	[Tx,x]_p&=&a|x_1|^p+(b_1+ib_2)x_2\overline{x_1}|x_1|^{(p-2)}+(c_1+ic_2)x_1\overline{x_2}|x_2|^{(p-2)}-a|x_2|^p \\
	&=& ar^p-as^p+(b_1+ib_2)rsr^{(p-2)}e^{i\phi}+(c_1+ic_2)rs s^{(p-2)}e^{-i\phi},
	\end{eqnarray*}
 where $\phi=\phi_{2}-\phi_{1}.$ Let $[Tx,x]_p=h+ik.$
Then \begin{eqnarray*}
 h&=&a(r^p-s^p)+rs(b_1 r^{p-2} +c_1 s^{p-2})\cos\phi+rs(-b_2 r^{p-2} +c_2 s^{p-2})\sin\phi \\
 k&=&rs(b_2 r^{p-2} +c_2 s^{p-2})\cos\phi+rs(b_1 r^{p-2} -c_1 s^{p-2})\sin\phi.
\end{eqnarray*}
Eliminating $\phi$ from $h$ and $k,$ we get a ellipse of the form\\
\begin{eqnarray}
\bigg(\frac{h-H}{F}\bigg)^2+\bigg(\frac{k}{G}\bigg)^2=1,
\end{eqnarray}\\
where
\begin{eqnarray*}
 H&=& a(r^p -s^p)\\
\lambda&=&|b|=|c|\\
 F&=&\frac{\lambda^2 rs(r^{2(p-2)}-s^{2(p-2)})}{[(b_1 r^{p-2} -c_1 s^{p-2})^2+(b_2 r^{p-2} +c_2 s^{p-2})^2]^{\frac{1}{2}}}\\
 G&=&\frac{\lambda^2 rs(r^{2(p-2)}-s^{2(p-2)})}{[(b_1 r^{p-2} +c_1 s^{p-2})^2+(-b_2 r^{p-2} +c_2 s^{p-2})^2]^{\frac{1}{2}}}.
\end{eqnarray*}
It is easy to see that $H, F$ and $G$ are continuous function of $r,s.$
Thus if we define
\[E(r,s):=\bigg\{(x,y):\bigg(\frac{x-H}{F}\bigg)^2+\bigg(\frac{y}{G}\bigg)^2=1\bigg\},\]
then $V (T)=\bigcup_{r^p+s^p=1}E(r,s)$.\\
Therefore, $\big(H+F \cos\psi , G\sin\psi\big)\in V (T)$ for all $\psi \in [0,2\pi]$.\\
Since $\lambda\neq 0$ and $p\neq 2$, then clearly $V (T) \not\subset \mathbb{R}$.\\
Now, since $G$ is a continuous function on a compact set, it attains its supremum. Let $G$ attains it supremum at $(r_0, s_0)$ and its supremum is $G_0.$
Clearly, $r_0s_0\neq 0$ and $r_0\neq s_0.$ Let the value of $H $ and $F$ at $(r_0,s_0)$ is $H_0$ and $F_0,$ respectively. 
Now $(H_0+F_0 \cos\frac{\pi}{2},G_0\sin\frac{\pi}{2})\in V (T)$, i.e.,  $(H_0,G_0)\in V (T).$
Then it is easy to see that\\
$(-H_0,G_0)\in V (T)$.
Since $H_0\neq 0$ and $F_0\neq 0$, we have
\[\bigg(\frac{0-H}{F}\bigg)^2+\bigg(\frac{G_0}{G}\bigg)^2>1~~ \mbox{for}~~ \mbox{all}~~ r,s \geq 0.\]
This shows that $(0,G_0)\notin V (T)$.
Hence  $V (T)$ is not convex. 
\end{proof}

Now, we give a relation between $V(T)$ and $V(T^t)$ considering $T, T^t \in \mathbb{L}(\ell_p^2)$. For this purpose we need the following lemma.
\begin{lemma}\label{lemma2}
Let $T\in \mathbb{L}(\ell^2_p)$ be such that $T = \begin{pmatrix}
	a & b  \\
	c & d \\
	\end{pmatrix}$, where $a,b,c,d \in \mathbb{R}$. Then $V(T)$ is symmetric about $X$-axis.
\end{lemma}
\begin{proof}
Let $x=(x_1,x_2)\in S_{\ell^2_p}$ with  
	$x_1=(\cos\theta)^{\frac{2}{p}} e^{i\phi_1},x_2=(\sin\theta)^{\frac{2}{p}} e^{i\phi_2},~~\mbox{where}~\theta,\phi_1,\\\phi_2 \in [0,2\pi].$
Therefore,  $[Tx,x]_p=h(\theta)+F(\theta)\cos\phi+iG(\theta)\sin\phi~,$
where
\begin{eqnarray*}
\phi&=&\phi_2-\phi_1\\
h(\theta)&=& a(\cos\theta)^2+d(\sin\theta)^2\\
F(\theta)&=&b(\cos\theta)^{\frac{2}{q}}(\sin\theta)^{\frac{2}{p}}+c(\cos\theta)^{\frac{2}{p}}(\sin\theta)^{\frac{2}{q}}\\
G(\theta)&=&b(\cos\theta)^{\frac{2}{q}}(\sin\theta)^{\frac{2}{p}}-c(\cos\theta)^{\frac{2}{p}}(\sin\theta)^{\frac{2}{q}}.
\end{eqnarray*}
Let $E_\theta=\{(x,y):\big(\frac{x-h}{F}\big)^2+\big(\frac{y}{G}\big)^2=1\}.$\\
Therefore, $V(T) $ is the union of the ellipses $E_{\theta}$ with center $(h,0)$, major-axis $F$ and minor-axis $G.$
Hence $V(T)$ is symmetric about $X$-axis.
\end{proof}
Let us relate the numerical range of $T$ and $T^t$ by the following theorem.
\begin{theorem}
	Let $T\in \mathbb{L}(\ell^2_p)$ be such that $T = \begin{pmatrix}
	a & b  \\
	c & d \\
	\end{pmatrix},$ where $a,b,c,d \in \mathbb{R}$. Then $V(T^t)$ is the mirror image of $V(T)$ with respect to the line $x=\frac{a+d}{2}.$ 
\end{theorem}
\begin{proof}
	Let $x=(x_1,x_2)\in S_{\ell^2_p}$. Now,
	\begin{eqnarray*}
		[Tx,x]_p&=&a|x_1|^p+d|x_2|^p+bx_2\overline{x_1}|x_1|^{p-2}+cx_1\overline{x_2}|x_2|^{p-2} \\
		&=&a|x_1|^p+d|x_2|^p+b|x_2||x_1|^{p-1}e^{i\phi}+c|x_1||x_2|^{p-1}e^{-i\phi} ,
	\end{eqnarray*}
	where $\phi=\arg(x_2)-\arg(x_1).$\\
	Let $(\alpha,\beta)\in V(T)$. Since $V(T)$ is symmetric about X-axis by Lemma \ref{lemma2}, if we able to show $(2\lambda_0-\alpha,-\beta)\in V(T^t)$,  where $\lambda_0=\frac{a+d}{2}$ and vise-versa then we are done. Using Proposition \ref{prop}, $(2\lambda_0-\alpha,-\beta)\in V(T^t)$ equivalent to $ (\alpha,\beta)\in V(2\lambda_0 I-T^t) $.\\
	Now, \[2\lambda_0 I-T^t=\begin{pmatrix}
	d & -c  \\
	-b & a \\
	\end{pmatrix}
	=T^0(\mbox{say}).\]
	Since $(\alpha,\beta)\in V(T),$ we have
	\begin{eqnarray*}
		\alpha&=&a|x_1|^p+d|x_2|^p+b|x_2||x_1|^{p-1}\cos\phi+c|x_1||x_2|^{p-1}\cos\phi\\
		\beta&=&b|x_2||x_1|^{p-1}\sin\phi-c|x_1||x_2|^{p-1}\sin\phi,
	\end{eqnarray*}
	for some $x=(x_1,x_2)\in S_{\ell^2_p}$ and $\phi=\arg(x_2)-\arg(x_1).$
	Let $y=(y_1,y_2)\in S_{\ell^2_p}.$  Then 
	\begin{eqnarray*}
		[T^0y,y]_p &=&d|y_1|^p+a|y_2|^p-by_1\overline{y_2}|y_2|^{p-2}-cy_2\overline{y_1}|y_1|^{p-2} \\
		&= &d|y_1|^p+a|y_2|^p-b|y_1||y_2|^{p-1}e^{-i\theta}-c|y_2||y_1|^{p-1}e^{i\theta},
	\end{eqnarray*} 
	where $\theta=\arg(y_2)-\arg(y_1)$.
	
	Let $(\alpha',\beta')\in V(T^0) .$ Then 
	\begin{eqnarray*}
	\alpha'&=&d|y_1|^p+a|y_2|^p-b|y_1||y_2|^{p-1}\cos\theta-c|y_2||y_1|^{p-1}\cos\theta \\
	\beta'&=&b|y_1||y_2|^{p-1}\sin\theta-c|y_2||y_1|^{p-1}\sin\theta. 
	\end{eqnarray*}
	Now, if we choose $y_1=x_2$ and $y_2=x_1e^{i\pi}.$
	Then $\theta=\arg(y_2)-\arg(y_1)=\pi+\arg(x_1)-\arg(x_2)=\pi-\phi$.
Therefore,
	\begin{eqnarray*}
		\alpha'&=&a|x_1|^p+d|x_2|^p-b|x_2||x_1|^{p-1}\cos(\pi-\phi)-c|x_1||x_2|^{p-1}\cos(\pi-\phi) \\
		&=& a|x_1|^p+d|x_2|^p+ b|x_2||x_1|^{p-1}\cos\phi+c|x_1||x_2|^{p-1}\cos\phi = \alpha \\  
		\beta'&=&b|x_2||x_1|^{p-1}\sin(\pi-\phi)-c|x_1||x_2|^{p-1}\sin(\pi-\phi) \\
		&=& b|x_2||x_1|^{p-1}\sin\phi-c|x_1||x_2|^{p-1}\sin\phi = \beta.
	\end{eqnarray*}
	Similarly, we can show that for every $(\alpha,\beta)\in V(T^t)$ imply $(2\lambda_0-\alpha,-\beta)\in V(T).$
	This completes the proof of the theorem.
\end{proof}

Next, we obtain the relation between $V(T)$ and $V(T^t)$ considering $ T\in \mathbb{L}(\ell_p^2)$ and $ T^t\in \mathbb{L}(\ell_q^2)$.

\begin{theorem}\label{T^t}
	Let $T\in \mathbb{L}(\ell^2_p)$ and $T^t\in \mathbb{L}(\ell^2_q).$ Then $V(T)=V(T^t)$.
\end{theorem}

\begin{proof}
	With out loss of generality, we may assume that $T\in \mathbb{L}(\ell^2_p)$ is of the form 
	\[T=
	\begin{pmatrix}
	|a|e^{i\alpha} & |b|e^{i\beta} \\ 
	\rule{0em}{3ex}|c|e^{i\gamma} & \hphantom{-} |d|e^{i\delta}
	\end{pmatrix} ,\]
	where $\alpha, \beta,\gamma$ and $\delta \in [0, 2\pi].$  
	Then $T^t\in \mathbb{L}(\ell^2_q)$ is of the form 
	\[T^t =\begin{pmatrix}
		|a|e^{i\alpha} & |c|e^{i\gamma} \\ 
		\rule{0em}{3ex}|b|e^{i\beta} & \hphantom{-} |d|e^{i\delta}
	\end{pmatrix} .\]
Now, to show that $V(T)=V(T^t)$, it is sufficient to prove that for every $z\in S_{\ell_p^2}$ there exists $w\in S_{\ell_q^2}$ such that $[Tz,z]_p=[T^tw,w]_q$ and vice-versa. Let $z=(\cos^{\frac{2}{p}} \theta e^{i\theta_1},\sin^{\frac{2}{p}} \theta e^{i\theta_2})\in S_{\ell_p^2},$ where $\theta, \theta_1, \theta_2\in [0,2\pi].$ Then by simple calculation, it can be seen that\\
 $[Tz,z]_p= h +ik$, where
	\begin{eqnarray*}
	h&=& |a|\cos^2\theta \cos \alpha +|d|\sin^2\theta \cos\delta\\
	& &+|b|\sin^{\frac{2}{p}}\theta\cos^{\frac{2}{q}}\theta\cos(\beta+\theta_2-\theta_1 )	+|c|\sin^{\frac{2}{q}}\theta\cos^{\frac{2}{p}}\theta\cos(\gamma+\theta_1-\theta_2 )
	\end{eqnarray*}
	and
	\begin{eqnarray*}
k&=& |a|\cos^2\theta \sin \alpha +|d|\sin^2\theta \sin\delta\\ 
& &+|b|\sin^{\frac{2}{p}}\theta\cos^{\frac{2}{q}}\theta\sin(\beta+\theta_2-\theta_1 )	+|c|\sin^{\frac{2}{q}}\theta\cos^{\frac{2}{p}}\theta\sin(\gamma+\theta_1-\theta_2 ).\\
\end{eqnarray*}
\noindent
Now, if we choose $w\in S_{\ell_q^2}$ such that  $w=(\cos^{\frac{2}{q}} \theta e^{i\theta_2},\sin^{\frac{2}{q}} \theta e^{i\theta_1})$ then $[T^tw,w]_q= h +ik.$
Similarly, we can show that for every $w\in S_{\ell_q^2}$ there exists $z\in S_{\ell_p^2}$ such that $[T^tw,w]_q=[Tz,z]_p.$ This completes the proof of the theorem.
\end{proof}

Finally, we have the following relation between $V(T)$ and $V(T^*)$ considering $ T\in \mathbb{L}(\ell_p^2)$ and $ T^*\in \mathbb{L}(\ell_q^2),$ where $T^*$ is conjugate transpose of $T.$

\begin{theorem}\label{T^*}
	Let $T\in \mathbb{L}(\ell^2_p)$ and $T^*\in \mathbb{L}(\ell^2_q)$, where $\frac{1}{p}+\frac{1}{q}=1.$ Then $V(T)=\overline {V(T^*)}$.
\end{theorem}

\begin{proof}
	Proceeding similarly as  in the proof of Theorem \ref{T^t}, we can see that for $x= (\cos^{\frac{2}{p}} \theta e^{i\theta_1},\sin^{\frac{2}{p}} \theta e^{i\theta_2})\in S_{\ell_p^2},$ where $\theta, \theta_1, \theta_2 \in [0,2\pi]$ we have $y=(\cos^{\frac{2}{q}} \theta e^{i\theta_1},\sin^{\frac{2}{q}} \theta $ $ e^{i\theta_2})  \in S_{\ell_q^2}$ such that $[Tx,x]_p=\overline{[T^*y,y]_q}.$ Therefore, we conclude that $V(T)=\overline {V(T^*)}.$
\end{proof}
\begin{remark} We know that for $T\in \mathbb{L}(\mathbb{H}),  W(T)=\overline{W(T^*)} .$ If $p=q=2$ then Theorem \ref{T^*} shows that $W(T)=\overline{W(T^*)}.$ 
\end{remark}

\end{document}